\documentclass[review,12pt]{elsarticle}

\usepackage{amsmath,amssymb,amsthm}
\usepackage{times}
\newtheorem{theorem}{Theorem}

\newtheorem{lemma}{Lemma}
\newtheorem{remark}{Remark}

\journal{}

\begin{document}

\begin{frontmatter}



\title{A Strehl Version of Fourth Franel Sequence}


\author[label1]{Belbachir Hacène}
\address[label1]{USTHB; Faculty of Mathematics, RECITS Laboratory, Po. Box 32, El Alia, Bab-Ezzouar 16111, Algiers, Algeria}
\ead{hbelbachir@usthb.dz}

\author[label1]{Otmani Yassine}
\ead{yotmani@usthb.dz}

\begin{abstract}
We give a combinatorial identity related to the Franel numbers involving the sum of fourth power of binomial coefficients. Furthermore, investigating in J. Miki\'c's proof of the first Strehl Identity, we provide a combinatorial proof of this identity using the double counting argument.
\end{abstract}




\end{frontmatter}


\section{Introduction}
It is well know that
\begin{equation}
\sum_{k=0}^{n}\binom{n}{k}=2^n.
\end{equation}
Also, the sums of power two of binomial coefficients is given by Vandermonde Identity
\begin{equation}
\beta_{n}:=\sum_{k=0}^{n}\binom{n}{k}^2=\binom{2n}{n},
\end{equation}
it is established that the sequence $\left( \beta_{n}\right) _{n\in\mathbb{N}}$ satisfy the following recurrence relation (see \cite[Corollary 3.1]{belbachir2})

\begin{equation}
n\beta_{n}=2\left( 2n-1\right) \beta_{n-1}\text{ where }\beta_{0}=1.
\end{equation}
In 1894, Franel, \cite{franel1}, introduced the sequences $\left( f_{n}\right) _{n\in\mathbb{N}}$ as
\begin{equation}\label{1.1}
f_{n}:=\sum_{k=0}^{n}\binom{n}{k}^3,
\end{equation}
where he provided that $\left( f_{n}\right) _{n\in\mathbb{N}}$ verified the following recurrence
 relation
\begin{align}
\left( n+1\right)^{2} f_{n+1}=\left(7n^2+7n  +2\right)f_{n}+8n^{2}f_{n-1} \text{ where }f_{0}=1 \text{, }f_{1}=2.
\end{align}
 Later, in 1895, Franel, \cite{franel2}, shows that the sequence $$\varphi_{n}:=\sum_{k=0}^{n}\binom{n}{k}^{4},$$ satisfies a three-term recurrence formula as follow
\begin{equation}\label{1.3}
\left( n+1\right)^{3} \varphi_{n+1}=2\left(2n+1 \right) \left(3n^2+3n+1\right)\varphi_{n}+4n\left(4n-1 \right) \left(4n+1 \right) \varphi_{n-1},
\end{equation}
where $\varphi_{0}=1$, $\varphi_{1}=2$ (for more details about the recurrence of sums that involving power of binomial coefficient one can see \cite{cus}). In 1905, MacMahon, \cite{macmahon}, investigated the master Theorem where he found the following identity 
\begin{equation}\label{1.4}
\sum_{k=0}^{n}\binom{n}{k}^{3}x^ky^{n-k}=\sum_{k=0}^{\lfloor n/2\rfloor}\binom{n}{k}\binom{n-k}{k}\binom{n+k}{k}\left( xy\right)^k\left( x+y\right) ^{n-2k} .
\end{equation}
Recently, in 1993, through applying the Chu-Vandermonde convolution \cite{belbachir}, Strehl, \cite{strehl},
obtained the following identity 
\begin{equation}\label{1.5}
f_{n}=\sum_{k=0}^{n}\binom{n}{k}^2\binom{2k}{n},
\end{equation} 
by simplifying $\binom{2k}{n}\binom{n}{k}$ on Strehl Identity, we get 
\begin{equation}\label{1.6}
f_{n}=\sum_{k=0}^{\lfloor n/2\rfloor}\binom{n}{2k}\binom{2k}{k}\binom{2n-2k}{n-k},
\end{equation}
in which Gould, \cite{gould}, derived it using Carlitz formula \cite{carlitz}. Furthermore, Zhao, \cite{zhao}, presents a combinatorial proof of the equivalence between Formula~\eqref{1.4} (for $x=y=1$) and Identity~\eqref{1.6}, employing free 2-Motzkin paths. Also, J. Miki\'c et al. \cite{miki}, established a combinatorial proof of Identity~\eqref{1.5}, using double counting argument.\\
The purpose of this paper is to prove a new extension of Franel number $\varphi_{n}$. Furthermore, we develop the J. Miki\'c's argument, \cite{miki}, of the first Strehl Identity to prove Identity~\eqref{1.7}, below combinatorially.\\
The paper is structured as follow; in Section 2, we present our result, then in Section 3, we provide the combinatorial proof of Identity~\eqref{1.7}, bellow using double counting argument.

\section{Main Theorem} 
Let us star by following lemma.
\begin{lemma}[\cite{belbachir,gould2}]\rm Let $n,m$ be positive integers and $x$ be complex, we have
\begin{equation}\label{7}
\binom{x}{n}\binom{n}{m}=\binom{x}{m}\binom{x-m}{n-m},
\end{equation}
\begin{equation}\label{8}
\sum_{k=0}^{n}\binom{x}{k}\binom{m}{n-k}=\binom{x+m}{n},
\end{equation}
\begin{equation}\label{9}
\sum_{k=0}^{n}\binom{n}{i}\binom{m}{n-k}\binom{x+n-k}{n+m}=\binom{x}{n}\binom{x}{m}.
\end{equation}
\end{lemma}
\begin{remark}\rm
The Identity~\eqref{9}, is exactly the Riordan Identity see \cite[(6.45)]{gould2}. 
\end{remark}
\begin{theorem}\rm
Let $n$ be positive integer, the following identity holds true
\begin{equation}\label{1.7}
\varphi_{n}=\sum_{k=0}^{n}\binom{n}{k}^2\binom{2k}{n}\binom{2n-k}{n}.
\end{equation}
\end{theorem}
\begin{proof}\rm
Let $s$ be left hand side of~\eqref{1.7}.  From~\eqref{8}, we have
\begin{align*}
\binom{2k}{n}=\sum_{i=0}^{n}\binom{k}{i}\binom{k}{n-i},
\end{align*}
then
\begin{align*}
s&=\sum_{k=0}^{n}\binom {n}{k}^2\binom{2n-k}{n}\sum_{i=0}^{n}\binom{k}{i}\binom{k}{n-i},\\
&=\sum_{i=0}^{n}\binom{n}{i}\binom{n}{n-i}\sum_{k=i}^{n}\binom {n-i}{k-i}\binom{i}{k+i-n}\binom{2n-k}{n},\\
&=\sum_{i=0}^{n}\binom{n}{i}\binom{n}{n-i}\sum_{k=i}^{n}\binom {n-i}{k-i}\binom{i}{n-k}\binom{2n-k}{n},
\end{align*}
set $k-i=j$, we get
\begin{align*}
s=\sum_{i=0}^{n}\binom{n}{i}\binom{n}{n-i}\sum_{j=0}^{n-i}\binom {n-i}{j}\binom{i}{n-i-j}\binom{2n-i-j}{n},
\end{align*}
finally, we apply the Riordan Identity~\eqref{9}, where get to result.
\end{proof}
\begin{remark}\rm
Follow the same steps of precedent proof we get this general form for all complex $x$
\begin{equation}
\sum_{k=0}^{n}\binom{n}{k}^2\binom{x}{k}\binom{x}{n-k}=\sum_{k=0}^{n}\binom{x}{n-k}\binom{x+k}{k}\binom{2\left( n-k\right) }{n}\binom{n}{k}.
\end{equation}
As a consequence we get for $x=1/2$ and $x=-1/2$, the following identities
\begin{align}
&\sum_{k=0}^{n}\binom{n}{k}^2\binom{2k}{k}\binom{2\left( n-k\right) }{n-k}\frac{1}{\left( 2k-1\right) \left( 2\left( n-k\right)-1\right)}\\&=\sum_{k=0}^{n}\binom{2k}{k}\binom{2\left( n-k\right) }{n-k}\binom{2\left( n-k\right) }{n}\binom{n}{k}\frac{\left( 2k+1\right)\left( -1\right) ^{n-k+1} }{2\left( n-k\right)-1}\nonumber,
\end{align}
\begin{align}\label{13}
\sum_{k=0}^{n}\binom{n}{k}^2\binom{2k}{k}\binom{2\left( n-k\right) }{n-k}=\sum_{k=0}^{n}\binom{2k}{k}\binom{2\left( n-k\right) }{n-k}\binom{2\left( n-k\right) }{n}\binom{n}{k}\left( -1\right) ^{k},
\end{align}
where~\eqref{13}, is Domb numbers \cite{domb}. Furthermore, set $x=-n-1$, we get
\begin{equation}
\sum_{k=0}^{n}\binom{n}{k}^2\binom{n+k}{k}\binom{2n-k}{n}=\sum_{k=0}^{n}\binom{n}{k}^2\binom{2n-k}{n}\binom{2\left( n-k\right) }{n}.
\end{equation}

\end{remark}
\section{Combinatorial Proof}
Before staring the combinatorial proof of Identity~\eqref{1.7}, we need to remember some fundamental concepts.
\begin{itemize}
\item Let $n$ be non negative integer, $\left[n \right]$ denotes the set $\left\lbrace1,2\ldots,n \right\rbrace $, $\left[0\right]$ denotes the empty set $\emptyset$.
\item Let $A$ be finite set, we note by $\vert A\vert$ the cardinal of $A$.
\end{itemize}
Now we are already to start our combinatorial proof in two steps. Let $n\in\mathbb{N}$ and $Y$ be the set define as follow
\begin{align*}
Y:=\left\lbrace \left(A,B,C\right)\vert A,B\subset\left[ 2n\right] ,C\subset\left[ 3n\right]\setminus\left[ n\right],\vert A\vert=\vert B\vert=\vert C\vert=n,A\subset B\Delta\left[ n\right],\right.\\ \left.B\setminus\left[ n\right]\subset C  \right\rbrace .
\end{align*} 
First step. Suppose $k=\vert B\setminus\left[ n\right] \vert$, it is clear that $0\leq k\leq n$, so we can choose those elements from the set $ \left[2n\right]\setminus\left[ n\right]$ in $\binom{n}{k}$ ways. The remain elements $n-k$ belong to $B\cap\left[ n\right]$, so we can choose them in $\binom{n}{n-k}$ ways (see \cite{miki}). Since, the $k$ elements of $B\setminus\left[ n\right]$ must be already in $C$, then the remain elements $n-k$ can be choose from remain elements of the set $\left( \left[3n\right]\setminus\left[ n\right]\right)\setminus \left( B\setminus\left[ n\right]\right) $ in $\binom{2n-k}{n-k}$ ways. Since, $\vert B\Delta\left[ n\right]\vert=2k$, then we have $\binom{2k}{n}$ ways to choose the elements of the set $A$ (see \cite{miki}). We conclude that
\begin{equation}\label{2.1}
\vert Y\vert=\sum_{k=0}^{n}\binom{n}{k}\binom{n}{n-k}\binom{2k}{n}\binom{2n-k}{n-k}.
\end{equation}
Second Step. Observe that
\begin{align*}
Y=\left\lbrace \left(A,B,C\right)\vert A,B\subset\left[ 2n\right] ,C\subset\left[ 3n\right]\setminus\left[ n\right],\vert A\vert=\vert B\vert=\vert C\vert=n,A\setminus\left[ n\right]\subset B,\right.\\ \left.B\setminus\left[ n\right]\subset C,  A\cap B\cap C\cap\left[n \right]=\emptyset   \right\rbrace.
\end{align*}

Let us star by counting the number of elements of $A$. Suppose $k=\vert A\cap\left[n \right]  \vert$. Obviously, $0\leq k\leq n$, we choose those elements from the set $\left[n \right]$ in $\binom{n}{k}$ ways, the remain $n-k$ elements belong to $A\setminus\left[n \right]$ and we choose them from the set $\left[2n \right]\setminus\left[n \right]$ in $\binom{n}{n-k}$ ways. In addition, these elements are in the set $B$, so the remain $k$ elements can be choose from $\left[2n \right]\setminus A$ in $\binom{n}{k}$ ways (for more details see \cite{miki}). Clearly $A\setminus\left[ n\right]\subset B\setminus\left[ n\right]$, then the $n-k$ elements of $A\setminus\left[ n\right]$ must already be in $C$. The remain elements of $C$ can be choose from the set $\left[3n\right] \setminus\left[ 2n\right]$ in $\binom{n}{k}$ ways. We conclude that
\begin{equation}\label{2.2}
\vert Y\vert=\sum_{k=0}^{n}\binom{n}{k}\binom{n}{n-k}\binom{n}{k}\binom{n}{k}.
\end{equation} 
From Identity~\eqref{2.1} and Identity~\eqref{2.2}, the proof of~\eqref{1.7} is done.\\

\begin{remark}\rm
Motivated by J. Miki\'c’s observation about generalization of Identity~\eqref{1.5}, see \cite[Remark 1]{miki}, we generalize the set $Y$ as follow
\begin{align*}
Y:=\left\lbrace \left(A,B,C\right)\vert A,B\subset\left[ m\right] ,C\subset\left[ m+n\right]\setminus\left[ n\right],\vert A\vert=\vert B\vert=\vert C\vert=n\right.\\ \left.,A\subset B\Delta\left[ n\right],B\setminus\left[ n\right]\subset C  \right\rbrace ,
\end{align*}
where by double counting argument we get
\begin{align}
\sum_{k=0}^{\min(m-n,n)}\binom{m-n}{k}\binom{n}{n-k}\binom{2k}{n}\binom{m-k}{n-k}\\=\sum_{k=0}^{\min(m-n,n)}\binom{n}{k}\binom{m-n}{n-k}\binom{m-n}{k}\binom{n}{k}.\nonumber
\end{align}
\end{remark}
 
\section{Acknowledgment}
This paper is partially supported by DGRSDT grant $n^{\circ}$ C0656701.

\bibliographystyle{elsarticle-num} 

\begin{thebibliography}{00}

\bibitem{belbachir} H. Belbachir, A combinatorial contribution to the multinomial Chu-Vandermonde convolution, {\it Ann. RECITS}. {\bf 1} (2014) 27--34.

\bibitem{belbachir2}H. Belbachir, A. Mehdaoui, Recurrence relation associated with the sums of square binomial coefficients, accepted in \emph{ Quaestiones Mathematicae}.

\bibitem{carlitz}L. Carlitz, Problem 352,
{\it Math. Mag}., {\bf 32} (1958) 47--48.

\bibitem{cus}T. W. Cusick, Recurrences for sums of powers of binomial coefficients, {\it J. Combin. Th.
Set. A}., {\bf 52} (1989) 77--83.

\bibitem{domb}C. Domb, On the theory of cooperative phenomena in crystals, {\it Adv. in Phys.}, {\bf 9} (1960) 149--361.

\bibitem{gould2}H. W. Gould, {\it Combinatorial Identities: A standardized set of tables listing 500 binomial coefficient summations}, Morgantown, W Va 1972.

\bibitem{gould}H. W. Gould, Sums of powers of binomial coefficients via Legendre polynomials Part 2, {\it Ars Combin}., {\bf 86} (2008) 161--173.

\bibitem{franel1}J. Franel, On a question of Laisant, {\it L’intermédiaire des Mathématiciens}. {\bf 1} (1894) 45--47.

\bibitem{franel2}J. Franel, On a question of J. Franel, {\it L’intermédiaire des Mathématiciens}. {\bf 2} (1895) 33--35.

\bibitem{macmahon}P. A. MacMahon, The sums of the powers of the binomial coefficients. {\it Quart. J. Math}., {\bf 32} (1902) 274--88.

\bibitem{miki}J. Mikić, J. SŠC, and J. Cvijic, A combinatorial proof of the first Strehl Identity. \texttt{http://www.ssmrmh.ro/wp-content/uploads/2019/05/\\A-COMBINATORIAL-PROOF-OF-FIRST-STREHL-IDENTITY.pdf} 

\bibitem{strehl} V. Strehl, Binomial identities—combinatorial and algorithmic aspects, {\it Discrete Math}., {\bf 136} (1994) 309--346.

\bibitem{zhao} A. F. Zhao, A combinatorial proof of two equivalent identities by free 2-Motzkin paths. {\it Integers}, 13, (2013).
\end{thebibliography}


\end{document}